\documentclass[amstex,12pt, amssymb]{article}

\usepackage{mathtext}
\usepackage[cp1251]{inputenc}
\usepackage[T2A]{fontenc}
\usepackage[dvips]{graphicx}
\usepackage{amsmath}
\usepackage{amssymb}
\usepackage{amsxtra}
\usepackage{latexsym}
\usepackage{ifthen}

\newcounter{lemma}[section]

\newcounter{corollary}[section]

\newcounter{remark}[section]

\newcounter{theorem}[section]

\newcounter{proposition}[section]

\numberwithin{equation}{section}

\def\Xint#1{\mathchoice
   {\XXint\displaystyle\textstyle{#1}}%
   {\XXint\textstyle\scriptstyle{#1}}%
   {\XXint\scriptstyle\scriptscriptstyle{#1}}%
   {\XXint\scriptscriptstyle\scriptscriptstyle{#1}}%
   \!\int}
\def\XXint#1#2#3{{\setbox0=\hbox{$#1{#2#3}{\int}$}
     \vcenter{\hbox{$#2#3$}}\kern-.5\wd0}}
\def\dashint{\Xint-}

\def\cc{\setcounter{equation}{0}
\setcounter{figure}{0}\setcounter{table}{0}}

\begin{document}

\markboth{\centerline{V. RYAZANOV AND E. SEVOST'YANOV}}
{\centerline{ON CONVERGENCE AND COMPACTNESS OF SPACE
HOMEOMORPHISMS}}

\author{{V. RYAZANOV AND E. SEVOST'YANOV}\\}

\title{
{\bf ON CONVERGENCE AND COMPACTNESS \\ OF SPACE HOMEOMORPHISMS}}

\maketitle

\large \begin{abstract} Various theorems on convergence of general
space homeomorphisms are proved and, on this basis, theorems on
convergence and compactness for classes of the so-called ring
$Q$--ho\-me\-o\-mor\-phisms are obtained. In particular, it was
established by us that a family of all ring
$Q$--ho\-me\-o\-mor\-phisms $f$ in ${\Bbb R}^n$ fixing two points is
compact provided that the function $Q$ is of finite mean
oscillation. These results will have wide applications to Sobolev's
mappings.
\end{abstract}


\large \cc
\section{Introduction}
We give here the foundations of the convergence theory for general
ho\-me\-o\-mor\-phisms in space and then develope the compactness
theory for the so--called $Q$--ho\-me\-o\-mor\-phisms. The ring
$Q$--ho\-me\-o\-mor\-phisms have been introduced first in the plane
in the connection with the study of the degenerate Beltrami
equations, see e.g. the papers \cite{RSY$_1$}--\cite{RSY$_5$} and
the monographs \cite{GRSY} and \cite{MRSY}. The theory of ring
$Q$--homeo\-mor\-phisms is applicable to various classes of mappings
with finite distortion intensively investigated in many recent
works, see e.g. \cite{MRSY} and \cite{KRSS} and further references
therein. The present paper is a natural continuation of our previous
works \cite{RS$_1$}--\cite{RS$_2$}.

\medskip

\newpage

$$
{}
$$

Given a family $\Gamma$ of paths $\gamma$ in ${\Bbb R}^n\,,$ $n\ge
2\,,$ a Borel function $\rho:{\Bbb R}^n \rightarrow [0,\infty]$ is
called {\bf admissible} for $\Gamma,$ abbr. $\rho \in {\rm
adm}\,\Gamma,$ if
$$\int\limits_{\gamma} \rho(x)\,|dx| \ge 1 $$
%
for each $\gamma\in\Gamma .$ The {\bf modulus} of $\Gamma$ is the
quantity
$$M(\Gamma) =\inf\limits_{ \rho \in {\rm adm}\,\Gamma}
\int\limits_{{\Bbb R}^n} \rho^n(x)\, dm(x) \,. $$
\medskip

Given a domain $D$ and two sets $E$ and $F$ in ${\overline{{\Bbb
R}^n}},$ \,\,$n\ge 2,$\,\, $\Gamma(E,F,D)$ denotes the family of all
paths $\gamma:[a,b] \rightarrow {\overline{{\Bbb R}^n}}$ which join
$E$ and $F$ in $D$, i.e., $\gamma(a) \in E, \ \gamma(b) \in F$ and
$\gamma(t) \in D$ for $a<t<b$. We set $\Gamma(E,F)=
\Gamma(E,F,{\overline{{\Bbb R}^n}})$ if $D={\overline{{\Bbb R}^n}}.$
A {\bf ring domain}, or shortly a {\bf ring} in $\overline{{\Bbb
R}^n}\,,$ is a domain $R$ in $\overline{{\Bbb R}^n}$ whose
complement has two connected components. Let $R$ be a ring in
$\overline{{\Bbb R}^n} .$ If $C_1$ and $C_2$ are the connected
components of $\overline{{\Bbb R}^n} \setminus R,$ we write
$R=R(C_1,C_2).$ The {\bf capacity} of $R$ can be defined by the
equlity
$${\rm cap}\,R(C_1, C_2)\ =\ M(\Gamma (C_1,C_2,R))\ , $$
see e.g. 5.49 in \cite{Vu}. Note also that
$$M(\Gamma (C_1,C_2,R))\ =\ M(\Gamma (C_1,C_2))\ , $$
 see e.g. Theorem 11.3 in \cite{Va}.
A {\bf conformal modulus} of a ring $R(C_0, C_1)$ is defined by
$${\rm mod\,}R(C_0, C_1)=\left(\frac{\omega_{n-1}}{M(\Gamma(C_0, C_1))}\right)^{1/(n-1)}\,,$$
where $\omega_{n-1}$ denotes the area of the unit sphere in ${\Bbb
R}^n,$ see e.g. (5.50) in \cite{Vu}.

\medskip
The following notion was motivated by the ring definition of
quasiconformality in \cite{Ge}. Let $D$ be a domain in ${{\Bbb
R}^n},$  $Q: D \rightarrow (0,\infty)$ be a (Lebesgue)
measurable function. Set

\newpage

$$A(x_0,r_1,r_2) = \{ x\in{\Bbb R}^n : r_1<|x-x_0|<r_2\}\
,$$
$$S(x_0,r_i) = \{ x\in{\Bbb R}^n : |x-x_0|=r_i\}\,\,,\ \ \ i=1,2.$$
%
We say, see \cite{RS$_1$} for the spatial case, that a homeomorphism
$f$ of $D$ into $\overline{{\Bbb R}^n}$ is a {\bf ring
$Q$--homeo\-mor\-phism at a point} $x_0\in D$ if
\begin{equation}\label{1.8}
M\left(\Gamma\left(f(S_1),\,f(S_2)\right)\right)\ \le
\int\limits_{A} Q(x)\cdot \eta^n(|x-x_0|)\ dm(x)
\end{equation}
for every ring $A=A(x_0,r_1,r_2),$ $0<r_1<r_2< r_0={\rm
dist}(x_0,\partial D),$ $S_i=S(x_0, r_i),$ $i=1,2,$ and for every
Lebesgue measurable function $\eta : (r_1,r_2)\rightarrow [0,\infty
]$ such that
$$\int\limits_{r_1}^{r_2}\eta(r)\ dr\ \ge\ 1\ .$$

If the condition (\ref{1.8}) holds at every point $x_0\in D,$ then
we also say that $f$ is a ring $Q$--homeo\-mor\-phism in the domain
$D.$

\cc
\bigskip
\section{On BMO and FMO functions}
\bigskip

\medskip
Recall that a real valued function $\varphi \in$ L$^1_{\rm loc}(D)$
given a domain $D\subset{{\Bbb R}}^n$ is said to be of {\bf bounded
mean oscillation} by John and Nierenberg,  abbr. $\varphi \in$
BMO(D) or simply $\varphi \in$ {\bf BMO}, if
$$\Vert \varphi \Vert _* = \sup\limits_{B\subset D}
\dashint\limits_B \vert \varphi (x) - \varphi _B \vert\, dm(x)
\,\,\, < \,\,\, \infty$$
where the supremum is taken over all balls $B$ in $D$
and
$$\varphi _B = \dashint \limits_B \varphi (x)\, dm(x):= {\frac{1}{
\vert B \vert}} \int\limits_B \varphi (x)\, dm(x) $$
is the average of the function $\varphi  $ over $B.$ For connections
of BMO functions with quasiconformal and quasiregular mappings, see
e.g. \cite{As}, \cite{AG}, \cite{Jo},\,\cite{MRV} and \cite{RR}.

Following \cite{IR}, we say that a function $\varphi : D\rightarrow
{\Bbb R} $ has {\bf finite mean oscillation} at a point $x_0 \in {D}
$ if

\begin{equation}\label{eq49} \overline{\lim\limits_{\varepsilon\rightarrow 0}}\ \
\dashint_{B( x_0 ,\varepsilon)}
|\varphi(x)-\widetilde{\varphi_{\varepsilon}}|\, dm(x)\, < \,\infty
\end{equation}
where
$$\widetilde{\varphi_{\varepsilon}}= \dashint_{B( x_0 ,\varepsilon)}
\varphi(x)\, dm(x)$$
%
is the average of the function $\varphi(x) $ over the ball
$B(x_0,\varepsilon)=\{x\in {\Bbb R}^n: |x-x_0|<\varepsilon\}.$ Note
that under $(\ref{eq49})$ it is possible that
$\widetilde{\varphi_{\varepsilon}} \rightarrow\infty $ as
$\varepsilon\rightarrow 0 .$

\medskip
We also say that a function $\varphi : D\rightarrow {\Bbb R} $ is of
finite mean oscillation in the domain D, abbr. $\varphi\in$ FMO$(D)$
or simply $\varphi\in$ {\bf FMO}, if $\varphi $ has finite mean
oscillation at every point $x \in {D}.$ Note that FMO is not
BMO$_{\rm loc},$ see examples in \cite{MRSY}, p. 211. It is
well--known that $L^{\,\infty }(D)\subset$ BMO$(D)\subset$ $L^p_{\rm
loc}(D)$ for all $1\le p<\infty ,$ see e.g.\,\,\cite{JN} and
\cite{RR}, but FMO$(D)\not\subseteq L_{\rm loc}^p(D)$ for any $p>1.$

\medskip
Recall some facts on finite mean oscillation from \cite{IR}, see
also 6.2 in \cite{MRSY}.

\medskip
\begin{proposition}{}\label{cor4}{\sl\,
 If, for some numbers $\varphi_{\varepsilon}\in
{{\Bbb R}},\ \ \varepsilon \in (0,\varepsilon_0] $,
$$\overline{\lim\limits_{\varepsilon\rightarrow 0}}\ \ \dashint_{B(
x_0 ,\varepsilon)} |\varphi(x)-\varphi_{\varepsilon}|\, dm(x) <
\infty\, ,$$
then  ${\varphi}$ has finite mean osillation at
$x_0$.}
\end{proposition}

\begin{corollary}{}\label{cor5}{\sl\,
If, for a point $x_0\in {D} ,$
$$\overline{\lim\limits_{\varepsilon\rightarrow 0}}\ \ \dashint_{B(
x_0 ,\varepsilon)} |\varphi(x)|\, dm(x) < \infty \, ,$$
then $\varphi $  has finite mean oscillation at
$x_0.$}
\end{corollary}

\medskip
\begin{lemma}{}\label{lem6}{\sl\,
Let $\varphi: D\rightarrow {{\Bbb R}}, n\ge 2 $, be a nonnegative
function with a finite mean oscillation at $0\in D $. Then
$$\int\limits_{\varepsilon<|x|< {\varepsilon_0}}\frac{\varphi (x)\,
dm(x)} {(|x| \log \frac{1}{|x|})^n} = O \left(\log{\log{
\frac{1}{\varepsilon}}}\right)$$
as $\varepsilon \rightarrow 0 $ for a positive $\varepsilon_0 \le
{\rm dist}\,(0,\partial D)$.\\}
\end{lemma}

This lemma takes an important part in many applications to the
mapping theory as well as to the theory of the Beltrami equations,
see e.g. the monographs \cite{MRSY} and \cite{GRSY}.

\cc
\section{Convergence of General Homeomorphisms}

In what follows, we use in $\overline{{{\Bbb R}}^n}={{\Bbb
R}}^n\bigcup\{\infty\}$ the {\bf spherical (chordal) metric}
$h(x,y)=|\pi(x)-\pi(y)|$ where $\pi$ is the stereographic projection
of $\overline{{{\Bbb R}}^n}$ onto the sphere
$S^n(\frac{1}{2}e_{n+1},\frac{1}{2})$ in ${{\Bbb R}}^{n+1},$ i.e.
$$h(x,y)=\frac{|x-y|}{\sqrt{1+{|x|}^2} \sqrt{1+{|y|}^2}},\ \,\, x\ne
\infty\ne y, $$
$$
h(x,\infty)=\frac{1}{\sqrt{1+{|x|}^2}}\ .
$$
It is clear that $\overline{{\Bbb R}^n}$ is homeomorphic to the unit
sphere ${\Bbb S}^n$ in ${\Bbb R}^{n+1}.$

The {\bf spherical (chordal) diameter} of a set $E \subset
\overline{{\Bbb R}^n}$ is
$$h(E)=\sup_{x,y \in E} h(x,y)\,.$$
 We also define $h(z, E)$
for $z\in \overline{{\Bbb R}^n}$ and $E\subseteq \overline{{\Bbb
R}^n}$ as a infimum of $h(z, y)$ over all $y\in E$ and $h(F, E)$ for
$F\subseteq  \overline{{\Bbb R}^n}$ and $E\subseteq \overline{{\Bbb
R}^n}$ as a infimum of $h(z, y)$ over all $z\in F$ and $y\in E.$
Later on, we also use the notation $B^{\,*}(x_0, \rho),$ $x_0\in
\overline{{\Bbb R}^n},$ $\rho\in (0, 1),$ for the balls $\{x\in
\overline{{\Bbb R}^n}: h(x, x_0)<\rho\}$ with respect to the
spherical metric.

\medskip
Let us start from the simple consequence of the well--known Brouwer
theorem on invariance of domains.

\medskip
\begin{corollary}\label{cor3.1}
{\sl\, Let $U$ be an open set in $\overline{{\Bbb R}^n}$ and let
$f:U\rightarrow \overline{{\Bbb R}^n}$ be continuous and injective.
Then $f$ is a homeomorphism $U$ onto $V=f(U).$}
\end{corollary}

\medskip
\begin{proof}
Let $y_0\in f(D)$ and $x_0:=f^{\,-1}(y_0).$ Set $B=B^{\,*}(x_0,
\varepsilon_0)$ where $0<\varepsilon_0<h(x_0, \partial D).$ Then
$\overline{B}\subset D.$ Note that a mapping
$f_0:=f|_{\overline{B}}$ is injective and continuous and maps the
compactum $\overline{B}$ into the Hausdorff topological space ${\Bbb
R}^n.$ Consequently, $f_0$ is a homeomorphism of $\overline{B}$ onto
the topological space $f_0(\overline{B})$ with the topology induced
by topology of ${\Bbb R}^n$ (see Theorem 41.III.3 in \cite{Ku$_2$}).
By the Brouwer theorem on invariance domains (see e.g. Theorem
4.7.16 in \cite{Sp}), $f$ maps the ball $B$ onto a domain in
$\overline{{\Bbb R}^n}$ as a homeomorphism. Hence it follows that
the mapping $f^{\,-1}(y)$ is continuous at the point $y_0$. Thus,
$f:D\rightarrow \overline{{\Bbb R}^n}$ is a homeomorphism.
\end{proof} $\Box$

\medskip
The {\bf kernel of a sequence of open sets}
$\Omega_l\subset\overline{{\Bbb R}^n}$, $l=1,2,\ldots$ is the open
set
$$
\Omega_0\ =\ {\rm Kern}\ \Omega_l\ \colon =
\bigcup\limits_{m=1}^{\infty}\ {\rm Int} \left(
\bigcap\limits_{l=m}^{\infty}\ \Omega_l\right),
$$
where ${\rm Int}\ A$ denotes the set consisting of all inner points
of $A$; in other words, ${\rm Int}\ A$ is the union of all open
balls in $A$ with respect to the spherical distance.

\medskip
The following statement for the plane case can be found in
\cite{BGR}, see also Proposition 2.7 in \cite{GRSY}.

\medskip
\begin{proposition} \label{prCOM3.1}
{\sl Let $g_l: D\rightarrow D^{\,\prime}_l,$
$D^{\,\prime}_l:=g_l(D),$ be a sequence of ho\-meo\-mor\-phisms
given in a domain $D\subset\overline{{\Bbb R}^n}.$ Suppose that
$g_l$ converge as $l\rightarrow\infty$ locally uniformly with
respect to the spherical (chordal) metric to a mapping $g:
D\rightarrow D^{\,\prime}:=g(D)\subset\overline{{\Bbb R}^n}$ which
is injective. Then $g$ is a homeomorphism and $D^{\,\prime}\subset
{\rm Kern}\, D^{\,\prime}_l.$}\end{proposition}

\medskip
\begin{proof} First of all, the mapping $g$ is continuous as a locally
uniform limit of continuous mappings, see e.g. Theorem 13.VI.3 in
\cite{Ku$_1$}. Thus, by Corollary \ref{cor3.1} $g$ is a
homeomorphism.

\medskip
Now, let $y_0$ be a point in $D^{\,\prime}.$ Consider the spherical
ball $B^*(z_0,\rho)$ where $z_0:=g^{\,-1}(y_0)\in D$ and $\rho<
h(z_0,
\partial D).$ Then
$$
r_0\ \colon =\ \min\limits_{z\in\partial B^*(z_0,\rho)}\ h(y_0,
g(z))\
>\ 0\ .
$$
There is an integer $N$ large enough such that $g_l(z_0)\in B^*(y_0,
r_0/2)$ for all $l\ge N$ and simultaneously
$$
B^{\,*}(y_0, r_0/2)\ \cap\ g_l(\partial B^{\,*}(z_0,\rho))\ =\
B^{\,*}(y_0, r_0/2)\ \cap\
\partial g_l(B^{\,*}(z_0,\rho))\ =\ \varnothing
$$
because $g_l\rightarrow g$ uniformly on the compact set $\partial
B^{\,*}(z_0,\rho).$ Hence by the con\-nec\-ted\-ness of balls
$$
B^{\,*}(y_0, r_0/2)\ \subset\ g_l(B^{\,*}(z_0,\rho))\qquad \forall\
l \ge N\,,
$$
see e.g. Theorem 46.I.1 in \cite{Ku$_2$}. Consequently, $y_0\in {\rm
Kern}\ D^{\,\prime}_l,$ i.e. $D^{\,\prime}\subset {\rm Kern}\
D^{\,\prime}_l$ by arbitrariness of $y_0.$ $\Box$\end{proof}

\medskip
\begin{remark}\rm\label{rmkCOM3.1} In particular, Proposition
\ref{prCOM3.1} implies that $D^{\,\prime}:=g(D)\subset{\Bbb R}^n$ if
$D_l^{\,\prime}:=g_l(D)\subset{\Bbb R}^n$ for all $l=1,2,\ldots
.$\end{remark}

\medskip
The following statement for the plane case can be found in the paper
\cite{KR}, see also Lemma 2.16 in the monograph \cite{GRSY}.

\medskip
\begin{lemma}\label{inverse}
{\sl Let $D$ be a domain in $\overline{{\Bbb R}^n},$ $l=1,2,\ldots,$
and let $f_l$ be a sequence of homeomorphisms from $D$ into
$\overline{{\Bbb R}^n}$ such that $f_l$ converge as
$l\rightarrow\infty$ locally uniformly with respect to the spherical
metric to a homeomorphism $f$ from $D$ into $\overline{{\Bbb R}^n}.$
Then $f_l^{-1}\rightarrow f^{-1}$ locally uniformly in $f(D)$, too.}
\end{lemma}

\begin{proof} Given a compactum $C\subset f(D),$ we have by Proposition
\ref{prCOM3.1} that $C\subset f_l(D)$ for all $l\ge l_0=l_0(C)$. Set
$g_l=f_l^{-1}$ and $g=f^{-1}$. The locally uniform con\-ver\-gen\-ce
$g_l\rightarrow g$ is equivalent to the so-called continuous
con\-ver\-gen\-ce, meaning that $g_l(u_l)\rightarrow g(u_0)$ for
every convergent sequence $u_l\rightarrow u_0$ in $f(D)$; see e.g.
\cite{Du}, p. 268 or Theorems 20.VIII.2 and 21.X.4 in \cite{Ku$_1$}.
So, let $u_l\in f(D)$, $l=0,1,2,\dots$ and $u_l\rightarrow u_0$ as
$l\rightarrow\infty$. Let us show that $z_l:=g(u_l)\rightarrow
z_0:=g(u_0)$ as $l\rightarrow\infty$.

It is known that every metric space is $\mathcal{L}^*$-space, i.e. a
space with a convergence (see, e.g., Theorem 21.II.1 in
\cite{Ku$_1$}), and the Uhryson axiom in compact spaces says that
$z_l\rightarrow z_0$ as $l\rightarrow\infty$ if and only if, for
every convergent subsequence $z_{l_k}\rightarrow z_*$, the equality
$z_*=z_0$ holds; see e.g. the definition 20.I.3 in \cite{Ku$_1$}.
Hence it suffices to prove that the equality $z_*=z_0$ holds for
every convergent subsequence $z_{l_k}\rightarrow z_*$ as
$k\rightarrow\infty$. Let $D_0$ be a subdomain of $D$ such that
$z_0\in D_0$ and $\overline{D_0}$ is a compact subset of $D$. Then
by Proposition \ref{prCOM3.1}, $f(D_0)\subset\text{Kern}f_l(D_0)$
and hence $u_0$ together with its neighborhood belongs to
$f_{l_k}(D_0)$ for all $k\ge K$. Thus, with no loss of generality we
may assume that $u_{l_k}\in f_{l_k}(D_0)$, i.e. $z_{l_k}\in D_0$ for
all $k=1,2,\dots$, and, consequently, $z_*\in D$. Then, by the
continuous convergence $f_l\rightarrow f$, we have that
$f_{l_k}(z_{l_k})\rightarrow f(z_*)$, i.e.
$f_{l_k}(g_{l_k}(u_{l_k}))=u_{l_k}\rightarrow f(z_*)$. The latter
implies that $u_0=f(z_*)$, i.e. $f(z_0)=f(z_*)$ and hence $z_*=z_0$.
The proof is complete. $\Box$\end{proof}

\medskip
The following statement for the plane case can be found in the paper
\cite{RSY$_5$}, see also Proposition 2.6 in the monograph
\cite{GRSY}.

\medskip
\begin{theorem}\label{pr1}
{\sl\, Let $D$ be a domain in $\overline{{\Bbb R}^n},$ $n\ge 2,$ and
let $f_m,$ $m=1,2,\ldots,$ be a sequence of homeomorphisms of $D$
into $\overline{{\Bbb R}^n}$ converging locally uniformly to a
discrete mapping $f: D\rightarrow \overline{{\Bbb R}^n}$ with
respect to the spherical metric. Then $f$ is a homeomorphism of $D$
into $\overline{{\Bbb R}^n}.$}
\end{theorem}

\medskip
\begin{proof}
First of all, let us show by contradiction that $f$ is injective.
Indeed, let us assume that there exist $x_1, x_2\in D,$ $x_1\ne
x_2,$ with $f(x_1)=f(x_2)$ and that $x_1\ne \infty.$ Set
$B_t=B(x_1, t).$
%
Let $t_0$ be such that $\overline{B_t}\subset D$ and $x_2\not\in
\overline{B_t}$  for every $t\in (0, t_0]$. By the Jordan--Brower
theorem, see e.g. Theorem 4.8.15 in \cite{Sp}, $f_m(\partial
B_t)=\partial f_m(B_t)$ splits $\overline{{\Bbb R}^n}$ into two
components
$$C_m:=f_m(B_t),\qquad C_m^{\,*}=\overline{{\Bbb R}^n}\setminus \overline{C_m}\,.$$
By construction $y_m:=f_m(x_1)\in C_m$ and $z_m:=f_m(x_2)\in
C_m^{\,*}.$ Remark that the ball $B^{\,*}(y_m, h(y_m, \partial
C_m))$ is contained inside of $C_m$ and, consequently, its closure
is inside of $\overline{C_m}.$ Hence
\begin{equation}\label{eq3}
h(y_m, \partial C_m)<h(y_m, z_m),\qquad m=1,2,\dots\,.
\end{equation}
By compactness of $\partial C_m=f_m(\partial B_t),$ there is
$x_{m,t} \in
\partial B_t$ such that
\begin{equation}\label{eq3A}
h(y_m, \partial C_m) = h(y_m, f_m(x_{m,t}))\,,\qquad m=1,2,\ldots\,.
\end{equation}
By compactness of $\partial B_t$, for every $t\in (0, t_0],$ there
is $x_t\in \partial B_t$ such that $h(x_{m_k,t}, x_t)\rightarrow 0$
as $k\rightarrow \infty$ for some subsequence $m_k.$ Since the
locally uniform convergence of continuous functions in a metric
space implies the continuous convergence (see \cite{Du}, p. 268 or
Theorem 21.X.3 in \cite{Ku$_1$}), we have that
$$h(f_{m_k}(x_{m_k,t}),f(x_{t})) \rightarrow 0 $$ as $k \rightarrow
\infty $. Consequently, from (\ref{eq3}) and (\ref{eq3A}) we obtain
that
$$ h(f(x_1),f(x_t))\le h(f(x_1),f(x_2))\qquad \forall\quad t\in(0, t_0]\,.$$
However, by the above assumption $f(x_1)=f(x_2)$ and we have
$f(x_t)=f(x_1)$ for every $t\in (0,t_0].$ The latter contradicts to
the discreteness of $f.$ Thus, $f$ is injective.

\medskip
It remains to show that $f$ and $f^{\,-1}$ are continuous. A mapping
$f$ is continuous as a locally uniform limit of continuous mappings,
see e.g. Theorem 13.VI.3 in \cite{Ku$_1$}. Finally, $f^{\,-1}$ is
continuous by Corollary \ref{cor3.1}.
\end{proof} $\Box$

\cc
\section{Convergence of Homeomorphisms with \\ Modular Conditions }
\label{sec4}

\medskip
Later on, the following lemma plays a very important role. Its plane
analog can be found in the paper \cite{BJ}, see also supplement A1
in the monograph \cite{GRSY}.

\medskip
\begin{lemma}\label{lem7.1}{\sl\, Let $f_m,$ $m=1,2,\ldots,$
be a sequence of homeomorphisms of a domain $D\subseteq {\Bbb R}^n$
into ${\Bbb R}^n,$ $n\ge 2,$ converging to a mapping $f$ uniformly
on every compact set in $D$ with respect to the spherical metric in
$\overline{{\Bbb R}^n}.$ Suppose that for every $x_0\in D$ there
exist sequences $R_k>0$ and $r_k\in (0, R_k),$ $k=1,2,\ldots,$ such
that $R_k\rightarrow 0$ as $k\rightarrow \infty$ and ${\rm
mod\,}f_m\left( A\left(x_0, r_k, R_k\right)\right)\rightarrow
\infty$ as $k\rightarrow \infty$ uniformly with respect to
$m=1,2,\ldots .$ Then the mapping $f$ is either a constant in
$\overline{{\Bbb R}^n}$ or a homeomorphism of $D$ into ${\Bbb
R}^n.$}
\end{lemma}

\medskip
\begin{proof} Assume that $f$ is not constant. Let us consider the
open set $V$ consisting of all points in $D$ which have
neighborhoods where $f$ is a constant and show that $f(x)\ne f(x_0)$
for every $x_0\in D\setminus V$ and $x\ne x_0.$ Without loss of
generality, we may assume that $f(x_0)\ne \infty.$ Now, let us fix a
point $x_*\ne x_0$ in $D\setminus V$ and choose $k=1,2,\ldots$ such
that $R:=R_k<|x_*-x_0|$ and
\begin{equation}\label{eq4.1}
{\rm mod}\, f_m\left(A(x_0, r, R)\right)>
\left(\omega_{n-1}/\tau_n(1)\right)^{1/(n-1)}
\end{equation}
for $r=r_k$ where $\tau_n(s)$ denotes the capacity of the
Tei\-ch\-m\"{u}\-l\-ler ring $R_{T, n}(s):=\left[{\Bbb R}^n\setminus
\{te_1: t\ge s\},\quad [-e_1, 0]\right],$ $s\in (0, \infty).$

\medskip
Let $c_m\in f_m(S(x_0, R))$ and $b_m\in f_m(S(x_0, r))$ be such that
$$\min\limits_{w\in f_m(S(x_0, R))}|w-f_m(x_0)|=|c_m-f_m(x_0)|\,,$$
$$\max\limits_{w\in f_m(S(x_0, r))}|w-f_m(x_0)|=|b_m-f_m(x_0)|\,.$$
Since $f_m$ is a homeomorphism, the set $f_m(A(x_0, r, R))$ is a
ring domain $\frak{R}_m=(C_m^1, C_m^2),$ where $a_m:=f_m(x_0)$ and
$b_m\in C_m^1,$ $c_m$ and $\infty\in C_m^2.$ Applying Lemma 7.34 in
\cite{Vu} with $a=a_m,$ $b=b_m$ and $c=c_m,$ we obtain that
\begin{equation}\label{eq7.1}
{\rm cap\,}\frak{R}_m=M(\Gamma(C_m^1, C_m^2))\ge \tau_n
\left(\frac{|a_m-c_m|}{|a_m-b_m|}\right)\,.
\end{equation}
Note that the function $\tau_n(s)$ is strictly decreasing (see Lemma
7.20 in \cite{Vu}). Thus, it follows from (\ref{eq4.1}) and
(\ref{eq7.1}) that
$$\frac{|a_m-c_m|}{|a_m-b_m|}\ge \tau_n^{\,-1}\left({\rm cap\,}
\frak{R}_m\right)>\tau_n^{\,-1}(\tau_n(1))=1\,.$$
Hence there is a spherical ring $A_m=\left\{y\in {\Bbb R}^n:
\rho_m<|y-f_m(x_0)|<\rho_m^{\,*}\right\}$ in the ring domain
$\frak{R}_m$ for every $m=1,2,\ldots\,.$ Since $f$ is not locally
constant at $x_0,$ we can find a point $x^{\,\prime}$ in the ball
$|x-x_0|<r$ with $f(x_0)\ne f(x^{\,\prime}).$ The ring $A_m$
separates $f_m(x_0)$ and $f_m(x^{\,\prime})$ from $f_m(x_*)$ and,
thus, $|f_m(x^{\,\prime})-f_m(x_0)|\le \rho_m$ and
$|f_m(x_*)-f_m(x_0)|\ge \rho_m^{\,*}.$ Consequently,
$|f_m(x^{\,\prime})-f_m(x_0)|\le |f_m(x_*)-f_m(x_0)|$ for all
$m=1,2,\ldots\,.$ Under $m\rightarrow\infty$ we have then
$0<|f(x^{\,\prime})-f(x_0)|\le |f(x_*)-f(x_0)|$ and hence $f(x_*)\ne
f(x_0).$

It remains to show that the set $V$ is empty. Let us assume that $V$
has a nonempty component $V_0.$ Then $f(x)\equiv z$ for every $x\in
V_0$ and some $z\in \overline{{\Bbb R}^n}.$ Note that $\partial
V_0\cap D\ne \varnothing$ by connectedness of $D$ because
$f\not\equiv const$ in $D$ and the set $D\setminus \overline{V_0}$
is also open. If $x_0\in
\partial V_0\cap D,$ then by continuity $f(x_0)=z$ contradicting the
first part of the proof because $x_0\in D\setminus V.$

\medskip
Thus, we have proved that the mapping $f$ is injective if $f$ is not
constant. But $f$ is continuous as a locally uniform limit of
continuous mappings $f_m,$ see Theorem 13.VI.3 in \cite{Ku$_1$}, and
then by Corollary \ref{cor3.1} $f$ is a homeomorphism. Finally, by
Remark \ref{rmkCOM3.1} $f(D)\subset{\Bbb R}^n$ and the proof is
complete.
\end{proof}$\Box$

\medskip
\begin{lemma}\label{lem7.2}
{\sl\,Let $D$ be a domain in ${\Bbb R}^n,$ $n\ge 2,$ $Q_m:
D\rightarrow (0, \infty)$ be measurable functions, $f_m,$
$m=1,2,\ldots,$ be a sequence of ring $Q_m$--ho\-meo\-mor\-phisms of
$D$ into ${\Bbb R}^n$ converging locally uniformly to a mapping $f.$
Suppose
\begin{equation}\label{eq6.3.24.A}
\int\limits_{\varepsilon<|x-x_0|<\varepsilon_0}
Q_m(x)\cdot\psi^n(|x-x_0|) \ dm(x)\ =\ o(I^{n}(\varepsilon,
\varepsilon_0))\quad\forall\ x_0\in D
\end{equation}
where $o(I^{n}(\varepsilon, \varepsilon_0))/I^{n}(\varepsilon,
\varepsilon_0)\to 0$ as $\varepsilon\rightarrow 0$ uniformly with
respect to $m$ for $\varepsilon_0<{\rm dist}\,(x_0,
\partial D)$ and a measurable function $\psi(t):(0,
\varepsilon_0)\rightarrow [0, \infty]$ such that
\begin{equation}\label{eq6.3.25A}
0<I(\varepsilon,
\varepsilon_0):=\int\limits_{\varepsilon}^{\varepsilon_0}\psi(t)\,\,dt
< \infty\quad\forall\,\varepsilon\in(0, \varepsilon_0)\ .
\end{equation}
Then the mapping $f$ is either a constant in $\overline{{\Bbb R}^n}$
or a homeomorphism into ${\Bbb R}^n.$}
\end{lemma}

\begin{remark}\label{remOS2.1.z} In particular, the conclusion of Lemma \ref{lem7.2} holds for
$Q$-ho\-meo\-mor\-phisms $f_m$  with a measurable function $Q:
D\rightarrow (0, \infty)$ such that
\begin{equation}\label{eq6.3.24}
\int\limits_{\varepsilon<|x-x_0|<\varepsilon_0}
Q(x)\cdot\psi^n(|x-x_0|) \ dm(x)=o(I^{n}(\varepsilon,
\varepsilon_0))\quad\quad\forall\ x_0\in D\ .
\end{equation}
\end{remark}

\begin{proof} By Lusin theorem there exists a Borel function $\psi_*(t)$
such that $\psi(t)=\psi_*(t)$ for a.e. $t\in (0, \varepsilon_0)$,
see e.g. 2.3.6 in \cite{Fe}. Since $Q_m(x)>0$ for all $x\in D$ we
have from (\ref{eq6.3.24.A}) that $I(\varepsilon, a)\rightarrow
\infty$ for every fixed $a\in (0, \varepsilon_0)$ and, in
particular, $I(\varepsilon, a)>0$ for every $\varepsilon\in (0, b)$
and some $b=b(a)\in (0, a).$ Given $x_0\in D$ and a sequence of such
numbers $b=\varepsilon_k\rightarrow 0$ as $k\rightarrow \infty,$
$k=1,2,\ldots ,$ consider a sequence of the Borel measurable
functions $\rho_{\varepsilon, k}$ defined as
$$\rho_{\varepsilon, k}(x)=
\left\{
\begin{array}{rr}
\psi_*(|x-x_0|)/I(\varepsilon, \varepsilon_k), &   \varepsilon<|x-x_0|<\varepsilon_k,\\
0,  &  {\rm otherwise}\,.
\end{array}
\right. $$
Note that the function $\rho_{\varepsilon, k}(x)$ is admissible for
$$\Gamma_{\varepsilon, k}:=\Gamma(S(x_0, \varepsilon), S(x_0,
\varepsilon_k), A(x_0, \varepsilon, \varepsilon_k) )$$
because
$$\int\limits_{\gamma}\rho_{\varepsilon, k}(x)|dx|\ge
\frac{1}{I(\varepsilon,
\varepsilon_k)}\int\limits_{\varepsilon}^{\varepsilon_k}\psi(t)dt=1$$
for all (locally rectifiable) curves $\gamma\in \Gamma_{\varepsilon,
k}$ (see Theorem 5.7 in \cite{Va}). Then by definition of ring
$Q$--homeomorphisms
\begin{equation}\label{eq7.2}
M(f_m(\Gamma_{\varepsilon, k}))\quad\le\quad
\frac{1}{I^n(\varepsilon, \varepsilon_k)} \int
\limits_{\varepsilon<|x-x_0|<\varepsilon_0}
Q(x)\cdot\psi^{n}(|x-x_0|)\,dm(x)
\end{equation}
for all $m\in {\Bbb N}.$ Note that $\frac{1}{I^n(\varepsilon,
\varepsilon_k)}=\alpha_{\varepsilon, k}\cdot
\frac{1}{I^{n}(\varepsilon, \varepsilon_0)},$ where
$\alpha_{\varepsilon, k}:=\left(1+\frac{I(\varepsilon_k,
\varepsilon_0)}{I(\varepsilon, \varepsilon_k)}\right)^n$ is
independent on $m$ and bounded as $\varepsilon\rightarrow 0.$ Then
it follows from (\ref{eq6.3.24.A}) and (\ref{eq7.2}) that there
exists $\varepsilon^{\,*}_k\in (0, \varepsilon_k)$ such that for all
%
$$M(f_m(\Gamma_{\varepsilon^{\,*}_k, k}))\quad\le\quad \frac{1}{2^{k}}\qquad \forall
\,m\in {\Bbb N}\,.$$
%
Applying Lemma \ref{lem7.1} we obtain a desired conclusion.
\end{proof} $\Box$

\medskip
The following important statements follow just from Lemma
\ref{lem7.2}.

\medskip
\begin{theorem}{}\label{th6.6.1A}{\sl\,
Let $D$ be a domain in ${\Bbb R}^n,$ $n\ge 2,$ $Q: D\rightarrow (0,
\infty)$ a Lebesgue measurable function and let $f_m,$
$m=1,2,\ldots,$ be a sequence of ring $Q$--ho\-me\-o\-mor\-phisms of
$D$ into ${\Bbb R}^n$ converging locally uniformly to a mapping $f.$
Suppose that $Q\in$ FMO. Then the mapping $f$ is either a constant
in $\overline{{\Bbb R}^n}$ or a homeomorphism into ${\Bbb R}^n.$}
\end{theorem}

\medskip
\begin{proof}
Let $x_0\in D.$ We may consider further that $x_0=0\in D.$ Choosing
a positive $\varepsilon_0<\min\left\{{\rm dist\,}\left(0, \,\partial
D\right),\,\, e^{\,-1}\right\},$ we obtain by Lemma \ref{lem6} for
the function $\psi(t)\,=\,\frac {1}{t\,\log{\frac1t}}$ that
$$\int\limits_{\varepsilon<|x|<\varepsilon_0}Q(x)\cdot\psi^n(|x|)
 \ dm(x)\,=\,O
\left(\log\log \frac{1}{\varepsilon}\right)\,.$$
%
Note that
$I(\varepsilon,
\varepsilon_0)\,:=\,\int\limits_{\varepsilon}^{\varepsilon_0}\psi(t)\,dt\,=\,
\log{\frac{\log{\frac{1}
{\varepsilon}}}{\log{\frac{1}{\varepsilon_0}}}}.$ Now the desired
conclusion follows from Lemma \ref{lem7.2}.
\end{proof}$\Box$

\medskip
The following conclusions can be obtained on the basis of Theorem
\ref{th6.6.1A}, Proposition \ref{cor4} and Corollary \ref{cor5}.

\medskip
\begin{corollary}\label{cor6.6.2A}{\sl\,
In particular, the limit mapping $f$ is either a constant in
$\overline{{\Bbb R}^n}$ or a homeomorphism of $D$ into
${\Bbb R}^n$ 
%
whenever
$$\overline{\lim\limits_{\varepsilon\rightarrow 0}}\ \
 \dashint_{B( x_0 ,\varepsilon)} Q(x)\ \ dm(x) <
 \infty\qquad\qquad\forall\,\,\,x_0
\in D 
$$
or  whenever every $x_0 \in D$ is a Lebesgue point of $Q.$}
\end{corollary}

\medskip
\begin{theorem}{}\label{th6.6.5A}{\sl\, Let $D$ be a domain in ${\Bbb
R}^n,$ $n\ge 2,$ and let $Q:D \rightarrow (0,\infty)$ be a
measurable function such that
 \begin{equation}\label{eq6.6.6}
\int\limits_{0}^{\varepsilon(x_0)}\frac{dr}{rq_{x_0}^{\frac{1}{n-1}}(r)}=\infty\quad
\quad \forall\,x_0\in D\end{equation} for a positive
$\varepsilon(x_0)< {\rm dist}\, (x_0,
\partial D)$ where $q_{x_0}(r)$ denotes the average of
$Q(x)$ over the sphere $|x-x_0|=r.$ Suppose that $f_m,$
$m=1,2,\ldots,$ is a sequence of ring $Q$--ho\-me\-o\-mor\-phisms
from $D$ into ${\Bbb R}^n$ converging locally uniformly to a mapping
$f.$ Then the mapping $f$ is either a constant in $\overline{{\Bbb
R}^n}$ or a homeomorphism into ${\Bbb R}^n.$ }
\end{theorem}

\medskip
\begin{proof}
Fix $x_0\in D$  and set
$I=I(\varepsilon, \varepsilon_0)=\int\limits
_{\varepsilon}^{\varepsilon_0}\psi(t)\,dt,$
$\varepsilon\in(0,\varepsilon_0)$, where
$$\psi(t)\quad=\quad \left \{\begin{array}{rr}
1/[tq^{\frac{1}{n-1}}_{x_0}(t)]\ , & \ t\in (\varepsilon,
\varepsilon_0)\ ,
\\ 0\ ,  &  \ t\notin (\varepsilon,
\varepsilon_0)\ .
\end{array} \right.$$
%
Note that $I(\varepsilon, \varepsilon_0)<\infty$ for every
$\varepsilon\in (0, \varepsilon_0).$ Indeed, by Theorem 3.15 in
\cite{RS$_1$} on the criterion of ring $Q-$homeomorphisms, we have
that
\begin{equation}\label{eq7.4}
M(f(\Gamma(S(x_0, \varepsilon), S(x_0, \varepsilon_0), A(x_0,
\varepsilon, \varepsilon_0))))\le\frac{\omega_{n-1}}{I^{n-1}}\,.
\end{equation}
On the other hand, by Lemma 1.15 in \cite{Na}, we see that
$$M(\Gamma(f(S(x_0, \varepsilon)), f(S(x_0, \varepsilon_0)), f(A(x_0, \varepsilon,
\varepsilon_0)))>0\,.$$  Then it follows from (\ref{eq7.4}) that
$I<\infty$ for every $\varepsilon\in (0, \varepsilon_0).$ In view of
(\ref{eq6.6.6}), we obtain that $I(\varepsilon, \varepsilon_*)>0$
for all $\varepsilon\in (0, \varepsilon_*)$ with some
$\varepsilon_*\in (0, \varepsilon_0).$ Finally, simple calculations
show that (\ref{eq6.3.24}) holds, in fact,
$$\int\limits_{\varepsilon<|x-x_0|<\varepsilon_*} Q(x)\cdot\psi^n(|x-x_0|)\
dm(x)\ =\ \omega_{n-1}\cdot I(\varepsilon, \varepsilon_*)$$ and
$I(\varepsilon,\varepsilon_*)=o\left(I^n(\varepsilon,\varepsilon_*)\right)$
by (\ref{eq6.6.6}). The rest follows by Lemma \ref{lem7.2}.
\end{proof}$\Box$

\medskip
\begin{corollary}\label{cor4.1}
{\sl \, In particular, the conclusion of Theorem \ref{th6.6.5A}
holds if $$q_{x_0}(r)=O\left(\log^{n-1}\frac{1}{r}\right)\qquad
\forall \,x_0\in D\,.$$}
\end{corollary}

\medskip
\begin{corollary}\label{cor6.6.7A}{\sl\, Under assumptions of
Theorem \ref{th6.6.5A}, the mapping $f$ is either a constant in
$\overline{{\Bbb R}^n}$ or a homeomorphism into ${\Bbb R}^n$
provided $Q(x)$ has singularities only of the logarithmic type of
the order which is not more than $n-1$ at every point $x_0\in D.$}
\end{corollary}

\medskip
\begin{theorem}\label{th4.1}{\sl\, Let $D$ be a domain in ${\Bbb
R}^n,$ $n\ge 2,$ and $Q:D \rightarrow (0,\infty)$ be a measurable
function such that
\begin{equation}\label{eq4.2}
\int\limits_{\varepsilon<|x-x_0|<\varepsilon_0}\frac{Q(x)}{|x-x_0|^n}\,dm(x)=
o\left(\log^n\frac{1}{\varepsilon}\right)\qquad\forall\, x_0\in D
\end{equation}
as $\varepsilon\rightarrow 0$ for some positive number
$\varepsilon_0=\varepsilon(x_0)<{\rm dist}\,(x_0,
\partial D).$
Suppose that $f_m,$ $m=1,2,\ldots,$ is a sequence of ring
$Q$--ho\-me\-o\-mor\-phisms from $D$ into $\overline{{\Bbb R}^n}$
converging locally uniformly to a mapping $f.$ Then the limit
mapping $f$ is either a constant in $\overline{{\Bbb R}^n}$ or a
homeomorphism into ${\Bbb R}^n.$
}
\end{theorem}

\medskip
\begin{proof} The conclusion follows from Lemma \ref{lem7.2} by the
choice $\psi(t)=\frac{1}{t}.$ \end{proof} $\Box$

\medskip
\medskip For every nondecreasing function $\Phi:[0,\infty ]\rightarrow
[0,\infty ] ,$ the {\bf inverse function} $\Phi^{-1}:[0,\infty
]\rightarrow [0,\infty ]$ can be well defined by setting
$$\Phi^{-1}(\tau)\ =\ \inf\limits_{\Phi(t)\ge \tau}\ t\ .$$
As usual, here $\inf$ is equal to $\infty$ if the set of
$t\in[0,\infty ]$ such that $\Phi(t)\ge \tau$ is empty. Note that
the function $\Phi^{-1}$ is nondecreasing, too. Note also that if
$h: [0,\infty ]\rightarrow [0,\infty ]$ is a sense--preserving
homeomorphism and $\varphi : [0,\infty ]\rightarrow [0,\infty ]$ is
a nondecreasing function, then
\begin{equation}\label{eq1!!!}
(\varphi\circ h)^{-1}\ =\ h^{-1}\circ\varphi^{-1} \ .
\end{equation}

\medskip

\begin{theorem}{}\label{th4.2} {\sl\,Let $D$ be a domain in ${\Bbb
R}^n,$ $n\ge 2,$ let $Q:D \rightarrow (0,\infty)$ be a measurable
function and $\Phi:[0, \infty]\rightarrow [0, \infty]$ be a
nondecreasing convex function. Suppose that
\begin{equation}\label{eq2!!}
\int\limits_D\Phi\left(Q(x)\right)\frac{dm(x)}{\left(1+|x|^2\right)^n}\
\le\ M<\infty
\end{equation}
and
\begin{equation}\label{eq3!}
\int\limits_{\delta}^{\infty}
\frac{d\tau}{\tau\left[\Phi^{-1}(\tau)\right]^{\frac{1}{n-1}}}\ =\
\infty
\end{equation}
for some $\delta>\Phi(0).$ Suppose that $f_m,$ $m=1,2,\ldots,$ is a
sequence of ring $Q$--ho\-me\-o\-mor\-phisms of $D$ into ${\Bbb
R}^n$ converging locally uniformly to a mapping $f.$ Then the
mapping $f$ is either a constant in $\overline{{\Bbb R}^n}$ or a
homeomorphism into ${\Bbb R}^n.$ }
\end{theorem}

\medskip
\begin{proof}
It follows from (\ref{eq2!!})--(\ref{eq3!}) and Theorem 3.1 in
\cite{RS$_2$} that the integral in (\ref{eq6.6.6}) is divergent for
some positive $\varepsilon(x_0)<{\rm dist}\,(x_0,
\partial D).$ The rest follows by Theorem \ref{th6.6.5A}.
\end{proof} $\Box$

\medskip

\begin{remark}\label{remark4.1} We may assume in Theorem \ref{th4.2}
that the function $\Phi(t)$ is not convex on the whole segment
$[0,\infty]$ but only on the segment $[t_*,\infty]$ where
$t_*=\Phi^{-1}(\delta)$. Indeed, every non-decreasing function $\Phi
:[0,\infty]\to[0,\infty]$ which is convex on the segment
$[t_*,\infty]$ can be replaced by a non-decreasing convex function
$\Phi_*:[0,\infty]\to[0,\infty]$ in the following way. Set
$\Phi_*(t)\equiv 0$ for $t\in [0,{t_*} ]$, $\Phi(t)=\varphi(t)$ for
$t\in[t_*,T_*]$ and $\Phi_*\equiv \Phi(t)$ for $t\in[T_*,\infty]$,
where $\tau =\varphi(t)$ is the line passing through the point
$(0,{t_*})$ and touching the graph of the function $\tau =\Phi(t)$
at a point $(T_*,\Phi(T_*))$, $T_*\in({t_*},\infty )$.  By the
construction we have that $\Phi_*(t)\le \Phi(t)$ for all
$t\in[0,\infty]$ and $\Phi_*(t)=\Phi(t)$ for all $t\ge T_*$ and,
consequently, the conditions (\ref{eq2!!}) and (\ref{eq3!}) hold for
$\Phi_*$ under the same $M$ and every $\delta > 0$.

Furthermore, by the same reasons it is sufficient to assume that the
function $\Phi$ is only minorised by a nondecreasing convex function
$\Psi$ on a segment $[T,\infty ]$ such that
\begin{equation}\label{eq!3!}
\int\limits_{\delta}^{\infty}
\frac{d\tau}{\tau\left[\Psi^{-1}(\tau)\right]^{\frac{1}{n-1}}}\ =\
\infty
\end{equation}
for some $T\in [0,\infty )$ and $\delta>\Psi(T).$ Note that the
condition (\ref{eq!3!}) can be written in terms of the function
$\psi(t)=\log \Psi(t)\ $:
\begin{equation}\label{eq3!!!33}
\int\limits_{\Delta}^{\infty}\ \psi(t)\ \frac{dt}{t^{n^{\prime}}}\
=\ \infty
\end{equation}
for some $\Delta>t_0\in [T,\infty ]$ where
$t_0:=\sup\limits_{\psi(t)=-\infty} t,$ $t_0=T$ if $\psi(T)>-\infty
,$ and where $\frac{1}{n'}+\frac{1}{n}=1$, i.e., $n'=2$ for $n=2$,
$n'$ is decreasing in $n$ and $n'=n/(n-1)\to1$ as $n\to\infty$, see
Proposition 2.3 in \cite{RS$_2$}. It is clear that if the function
$\psi$ is nondecreasing and convex, then the function
$\Phi=e^{\psi}$ is so but the inverse conlusion generally speaking
is not true. However, the conclusion of Theorem \ref{th4.2} is valid
if $\psi^m(t)$, $t\in [T,\infty ]$, is convex and (\ref{eq3!!!33})
holds for $\psi^m$ under some $m\in{\mathbb N}$ because $e^{\tau}\ge
\tau^m/m!$ for all $m\in{\mathbb N}$.
\end{remark}

\medskip

\begin{corollary}\label{cor6.6.7B}{\sl\,
In particular, the conclusion of Theorem \ref{th4.2} is valid if,
for some $\alpha > 0$,
\begin{equation}\label{eq2!!!2}
\int\limits_D e^{\alpha Q^{\frac{1}{n-1}}(x)
}\frac{dm(x)}{\left(1+|x|^2\right)^n}\ \le\ M<\infty \ .
\end{equation}
The same is true for any function $\Phi=e^{\psi}$ where $\psi(t)$ is
a finite product of the function $\alpha t^{\beta}$, $\alpha >0$,
$\beta \ge 1/(n-1)$, and some of the functions
$[\log(A_1+t)]^{\alpha_1}$, $[\log\log(A_2+t)]^{\alpha_2},\ \ldots \
$, $\alpha_m\ge -1$, $A_m\in{\mathbb R }$, $m\in{\mathbb N}$, $t\in
[T,\infty ]$, $\psi(t)\equiv\psi(T)$,  $t\in [0,T]$.}
\end{corollary}

\medskip

\begin{remark}\label{remark4.3}
For further applications, the integral conditions (\ref{eq2!!}) and
(\ref{eq3!}) for $Q$ and $\Phi$ can be written in other forms that
are more convenient for some cases. Namely, by (\ref{eq1!!!}) with
$h(t)=t^{\frac{1}{n-1}}$ and $\varphi(t)=\Phi(t^{n-1})$, $\Phi =
\varphi\circ h$, the couple of conditions (\ref{eq2!!}) and
(\ref{eq3!}) is equivalent to the following couple
\begin{equation}\label{eq2!!!}
\int\limits_D\varphi\left(Q^{\frac{1}{n-1}}(x)\right)\frac{dm(x)}{\left(1+|x|^2\right)^n}\
\le\ M<\infty
\end{equation}
and
\begin{equation}\label{eq3!!!}
\int\limits_{\delta}^{\infty} \frac{d\tau}{\tau\varphi^{-1}(\tau)}\
=\ \infty
\end{equation}
for some $\delta>\varphi(0).$ Moreover, by Theorem 2.1 in
\cite{RSY$_6$} the couple of the conditions (\ref{eq2!!!}) and
(\ref{eq3!!!}) is in turn equivalent to the next couple
\begin{equation}\label{eq2!!!2A}
\int\limits_D e^{\psi\left( Q^{\frac{1}{n-1}}(x)\right)
}\frac{dm(x)}{\left(1+|x|^2\right)^n}\ \le\ M<\infty
\end{equation}
and
\begin{equation}\label{eq3!!!3}
\int\limits_{\Delta}^{\infty}\ \psi(t)\ \frac{dt}{t^2}\ =\ \infty
\end{equation}
for some $\Delta>t_0$ where $t_0:=\sup\limits_{\psi(t)=-\infty}t,$
$t_0=0$ if $\psi(0)>-\infty .$

Finally, as it follows from Lemma \ref{lem7.2} all the results of
this section are valid if $f_m$ are $Q_m$-homeomorphisms and the
above conditions on $Q$ hold for $Q_m$ uniformly with respect to the
parameter $m=1,2,\ldots $.
\end{remark}

\cc
\section{On Completeness of Ring Homeomorphisms}

\medskip
The following result for the plane case can be found in the paper
\cite{RSY$_5$}, Theorem 4.1, see also Theorem 6.2 in the monograph
\cite{GRSY}.

\medskip
\begin{theorem}{}\label{thSTR4.2b}
{\sl Let $f_m:D\rightarrow\overline{{\Bbb R}^n},$ $m=1,2,\ldots\ ,$
be a sequence of ring $Q$--ho\-meo\-mor\-phisms at a point $x_0\in
D.$ If $f_m$ converges locally uniformly to a ho\-meo\-mor\-phism
$f:D\rightarrow\overline{{\Bbb R}^n} ,$ then $f$ is also a ring
$Q$--ho\-me\-o\-mor\-phism at $x_0.$} \end{theorem}

\medskip
\begin{proof} Note first that every point $w_0\in D^{\,\prime}=f(D)$
belongs to $D_m^{\,\prime}=f_m(D)$ for all $m\ge N$ together with
$\overline{B^{\,*}(w_0,\varepsilon)},$ where
$B^{\,*}(w_0,\varepsilon)=\{ w\in\overline{{\Bbb R}^n} :
h(w,w_0)<\varepsilon \}$ for some $\varepsilon > 0$ (see Proposition
\ref{prCOM3.1}).

Now, we note that $D^{\,\prime}=\bigcup_{l=1}^{\infty}\ C_l$ where
$C_l=\overline{D^*_l},$ and $D^*_l$ is a connected component of the
open set $\Omega_l = \{ w\in D^{\,\prime} : h(w,\partial
D^{\,\prime})>1/l \},$ $l=1,2,\ldots ,$ including a fixed point
$w_0\in D^{\,\prime}.$ Indeed, every point $w\in D^{\,\prime}$ can
be joined with $w_0$ by a path $\gamma$ in $D^{\,\prime}.$ Because
the locus $|\gamma |$ is compact we have that $h(|\gamma |, \partial
D^{\,\prime})>0$ and, consequently, $|\gamma|\subset D^{\,*}_l$ for
large enough $l=1,2,\ldots .$

Next, take an arbitrary pair of continua $E$ and $F$ in $D$ which
belong to the different connected components of the complement of a
ring $A=A(x_0,r_1,r_2)= \{x\in{\Bbb R}^n : r_1 < |x-x_0| < r_2 \},$
$x_0\in D,$ $0<r_1<r_2<r_0={\rm dist}\,(x_0, \partial D).$ For $l\ge
l_0,$ continua $f(E)$ and $f(F)$ belong to $D^{\,*}_l.$ Then the
continua $f_m(E)$ and $f_m(F)$ also belong to $D^{\,*}_l$ for large
enough $m.$ Fix one such $m.$ It is known that
$$
M(\Gamma(f_m(E),f_m(F), D^{\,*}_l))\ \rightarrow\
M(\Gamma(f(E),f(F), D^{\,*}_l))
$$
as $m\rightarrow\infty ,$ see e.g. Theorem A.12 of Section A1 in
\cite{MRSY}. However, $D^{\,*}_l\subset f_m(D)$ for large enough $m$
and hence
$$
M(\Gamma(f_m(E),f_m(F), D^{\,*}_l))\le M(\Gamma(f_m(E),f_m(F),
f_m(D)))
$$
and, thus, by definition of ring $Q$--ho\-me\-o\-mor\-phisms
$$
M(\Gamma(f(E), f(F), D^{\,*}_l))\le\int\limits_{A} Q(x)\cdot
\eta^n(|x-x_0|)\ dm(x)
$$
for every measurable function $\eta :(r_1,r_2)\rightarrow [0,\infty
]$ such that $\int\eta(r)\ dr\ge 1.$ Finally, since
$\Gamma=\bigcup\limits_{l=l_0}^{\infty}\Gamma_l$ where
$\Gamma=\Gamma(f(E),f(F), f(D)),$ $\Gamma_l=\Gamma(f(E),f(F),
D^{\,*}_l)$ is increasing in $l=1,2,\ldots ,$ we obtain that
$M(\Gamma)=\lim\limits_{l\rightarrow\infty}M(\Gamma_l)$ (see e.g.
Theorems A.7 and A.25 in \cite{MRSY}). Thus,
$$
M(\Gamma(f(E), f(F), f(D))\le \int\limits_{A} Q(x)\cdot
\eta^n(|x-x_0|)\ dm(x)\ ,
$$
i.e., $f$ is a ring $Q$--ho\-me\-o\-mor\-phism at $x_0.$
$\Box$\end{proof}

\cc
\section{Normal Classes of Ring $Q$--ho\-me\-o\-mor\-phisms}

Given a domain $D$ in ${\Bbb R}^n,$ $n\ge 2,$ a measurable function
$Q:D\rightarrow (0, \infty),$ and $\Delta>0,$ denote by ${\frak
F}_{Q, \Delta}$ the family of all ring $Q$--ho\-me\-o\-mor\-phisms
$f$ of $D$ into $\overline{{\Bbb R}^n}$ such that
$h\left(\overline{{\Bbb R}^n}\setminus f(D)\right)\ge\Delta.$ Recall
that a class of mappings is called normal if every sequence of
mappings in the class contains a subsequence that converges locally
uniformly.

\medskip
\begin{lemma}\label{lem7.3}
{\sl\, Let $D$ be a domain in ${\Bbb R}^n,$ $n\ge 2,$ and let $Q:
D\rightarrow (0, \infty)$ be a measurable function. If the
conditions (\ref{eq6.3.25A})--(\ref{eq6.3.24}) hold, then the class
${\frak F}_{Q, \Delta}$ forms a normal family for all $\Delta>0$. }
\end{lemma}

\medskip
\begin{proof}
By Lemma 7.5 in \cite{MRSY}, cf. also Lemma 4.1 in \cite{RS$_1$},
for $y\in B(x_0,r_0),$ $r_0< {\rm dist}\, (x_0,\partial D),$
$S_0=\{x\in{\Bbb R}^n: |x-x_0|=r_0\}$ and $S=\{x\in{\Bbb R}^n :
|x-x_0|=|y -x_0|\}$ we have that
\begin{equation}\label{eq7.6}
h(f(y), f(x_0))\ \le\ \frac{\alpha_n}{\Delta}\cdot
\hbox{exp}\left(-\left\{\frac{\omega_{n-1}}{M\left(\Gamma\left(f(S),\,f(S_0),
f(D)\right)\right)} \right\}^{1/n-1}\right)\ \end{equation}
where $\omega_{n-1}$ is the area of the unit sphere ${\Bbb S}^{n-1}$
in ${\Bbb R}^n$, $\alpha_n=2\lambda_n^2$ with $\lambda_n \in
[4,2e^{n-1}),$ $\lambda_2=4$ and $\lambda_n^{\frac 1n} \rightarrow
e$ as $n\rightarrow \infty.$
We may consider that $\psi$ is a Borel function, because, by 2.3.4
and 2.3.6 in \cite{Fe} there exists a Borel function $\psi_*(t)$
with $\psi(t)=\psi_*(t)$ for a.e. $t\in (0, \varepsilon_0).$ Given
$\varepsilon\in (0, \varepsilon_0),$ consider a Borel measurable
function $\rho_{\varepsilon}$ defined as
$$\rho_{\varepsilon}(x)=
\left\{
\begin{array}{rr}
\psi_*(|x-x_0|)/I(\varepsilon, \varepsilon_0), &   \varepsilon<|x-x_0|<\varepsilon_0,\\
0,  &  {\rm otherwise}\,.
\end{array}
\right. $$
Note that $\rho_{\varepsilon}(x)$ is admissible for
$\Gamma_{\varepsilon}:=\Gamma(S(x_0, \varepsilon), S(x_0,
\varepsilon_0), A(x_0, \varepsilon, \varepsilon_0))$ because
$$\int\limits_{\gamma}\rho_{\varepsilon}(x)|dx|\ge
\frac{1}{I(\varepsilon,
\varepsilon_0)}\int\limits_{\varepsilon}^{\varepsilon_0}\psi(t)dt=1$$
for all (locally rectifiable) curves $\gamma\in
\Gamma_{\varepsilon}$ (see Theorem 5.7 in \cite{Va}). Then by
definition of ring $Q$--homeomorphism at the point $x_0$
\begin{equation}\label{eq7.7}
 M(f(\Gamma_{\varepsilon}))\quad\le\quad
{\cal J}(\varepsilon):=\frac{1}{I^n(\varepsilon, \varepsilon_0)}
\int \limits_{\varepsilon<|x-x_0|<\varepsilon_0}
Q(x)\cdot\psi^{n}(|x-x_0|)\,dm(x)
\end{equation}
for all $f\in {\frak F}_{Q, \Delta}.$ It follows from
(\ref{eq6.3.24}) that, given $\sigma>0,$ there exists
$\delta=\delta(\sigma)$ such that ${\cal J}(\varepsilon)<\sigma$ for
all $\varepsilon\in (0, \delta).$ Then from (\ref{eq7.6}) and
(\ref{eq7.7}) we have that
\begin{equation}\label{eq7.8}
h(f(x), f(x_0))\ \le\ \frac{\alpha_n}{\Delta}\cdot
\hbox{exp}\left(-\left\{\frac{\omega_{n-1}}{\sigma}
\right\}^{1/n-1}\right)\ \end{equation} provided $|x-x_0|<\delta.$
In view of arbitrariness of $\sigma>0$ the equicontinuity of ${\frak
F}_{Q, \Delta}$ follows from (\ref{eq7.8}).
\end{proof} $\Box$

\medskip
\begin{remark}\label{remark6.1}
In particular, the conclusion of Lemma \ref{lem7.3} holds if at
least one of the conditions on $Q$ in Theorems
\ref{th6.6.1A}--\ref{th4.2} and Corollary
\ref{cor6.6.2A}--\ref{cor6.6.7B} holds. The corresponding normality
results have been formulated in \cite{RS$_1$} and \cite{RS$_2$} and
hence we will not repeat them  in the explicit form here.

Furthermore, as it follows from the analysis of the proof of Lemma
\ref{lem7.3}, its conlusion is valid for a more wide class ${\frak
F}_{\Delta}$, $\Delta > 0$, consisting of all ring
$Q$--ho\-me\-o\-mor\-phisms $f$ of $D$ into $\overline{{\Bbb R}^n}$
such that $h\left(\overline{{\Bbb R}^n}\setminus
f(D)\right)\ge\Delta$ satisfying the uniform condition
(\ref{eq6.3.24}) for the variable $Q$ but with a fixed function
$\psi$ in (\ref{eq6.3.25A}). Thus, the conclusion of Lemma
\ref{lem7.3} is also valid if at least one of the conditions on $Q$
in Theorems \ref{th6.6.1A}--\ref{th4.2} and Corollary
\ref{cor6.6.2A}--\ref{cor6.6.7B} is uniform with respect to the
variable functional parameter $Q$.

All notes in Remarks \ref{remark4.1} and \ref{remark4.3} are also
valid for the normality results.
\end{remark}

\medskip
\cc
\section{On Compact Classes of Ring $Q$--homeomorphisms}

\medskip
Given a domain $D$ in ${\Bbb R}^n,$ $n\ge 2,$ a measurable function
$Q:D \rightarrow (0, \infty),$ $x_1, x_2\in D,$ $y_1, y_2\in {\Bbb
R}^n,$ $x_1\ne x_2,$ $y_1\ne y_2,$ set $\frak{R}_{Q}$ the class of
all ring $Q$--ho\-me\-o\-mor\-phisms from $D$ into ${\Bbb R}^n,$
$n\ge 2,$ satisfying the normalization conditions $f(x_1)=y_1,$
$f(x_2)=y_2.$

\medskip
Recall that a class of mappings is called compact if it is normal
and closed. Combining the above results on normality and closeness,
we obtain the following results on compactness for the classes of
ring $Q$--homeomorphisms.

\medskip
\begin{theorem}{}\label{th6.6.1}{\sl\,
If $Q\in$ FMO, then the class $\frak{R}_{Q}$ is compact.}
\end{theorem}

\medskip
\begin{corollary}\label{cor6.6.2}{\sl\,
The class \, $\frak{R}_{Q}$ is compact if
$$\overline{\lim\limits_{\varepsilon\rightarrow 0}}\ \
 \dashint_{B( x_0 ,\varepsilon)} Q(x)\ \ dm(x) <
 \infty\qquad\qquad\forall\,\,\,x_0
\in D 
$$}
\end{corollary}

\medskip
\begin{corollary}\label{cor6.6.4} {\sl\, The class\,
$\frak{R}_{Q}$ is compact if every $x_0 \in D$ is a Lebesgue point
of $Q.$}
\end{corollary}

\medskip
\begin{theorem}{}\label{th6.6.5}{\sl\, Let
$Q$ satisfy the condition
$$
\int\limits_{0}^{\varepsilon(x_0)}\frac{dr}{rq_{x_0}^{\frac{1}{n-1}}(r)}=
\infty\qquad\forall\, x_0\in D$$ for some $\varepsilon(x_0)< {\rm
dist}\, (x_0,
\partial D)$ where $q_{x_0}(r)$ denotes the average of
$Q(x)$ over the sphere $|x-x_0|=r.$ Then the class $\frak{R}_{Q}$ is
compact.}
\end{theorem}

\medskip
\begin{corollary}\label{cor6.6.7}{\sl\, The class $\frak{R}_{Q}$
is compact if $Q(x)$ has singularities only of the logarithmic type
of the order which is not more than $n-1$ at every point $x_0\in
D$.}
\end{corollary}

\medskip
\begin{theorem}\label{th4.1A}{\sl\, The class\,
$\frak{R}_{Q}$ is compact if
$$\int\limits_{\varepsilon<|x-x_0|<\varepsilon_0}\frac{Q(x)}{|x-x_0|^n}\,dm(x)=
o\left(\log^n\frac{1}{\varepsilon}\right)\qquad \forall\,x_0\in D$$
as $\varepsilon\rightarrow 0$ for some
$\varepsilon_0=\varepsilon(x_0)<{\rm dist\,}(x_0, \partial D).$}
\end{theorem}

\medskip
\begin{theorem}{}\label{th4.2A}{\sl\,The class\,
$\frak{R}_{Q}$ is compact if
\begin{equation}\label{eq7.7A}
\int\limits_D\Phi\left(Q(x)\right)\frac{dm(x)}{\left(1+|x|^2\right)^n}\
\le\ M\ <\ \infty
\end{equation}
for a nondecreasing convex function $\Phi:[0, \infty]\rightarrow [0,
\infty]$ such that
\begin{equation}\label{eq7.77}
\int\limits_{\delta}^{\infty}
\frac{d\tau}{\tau\left[\Phi^{-1}(\tau)\right]^{\frac{1}{n-1}}}\ =\
\infty
\end{equation}
for some $\delta>\Phi(0).$}
\end{theorem}

\medskip

\begin{corollary}\label{cor6.6.7C}{\sl\,
In particular, the conclusion of Theorem \ref{th4.2A} is valid if,
for some $\alpha > 0$,
\begin{equation}\label{eq2!!!2c}
\int\limits_D e^{\alpha Q^{\frac{1}{n-1}}(x)
}\frac{dm(x)}{\left(1+|x|^2\right)^n}\ \le\ M<\infty \ .
\end{equation}
The same is true for any function $\Phi=e^{\psi}$ where $\psi(t)$ is
a finite product of the function $\alpha t^{\beta}$, $\alpha >0$,
$\beta \ge 1/(n-1)$, and some of the functions
$[\log(A_1+t)]^{\alpha_1}$, $[\log\log(A_2+t)]^{\alpha_2},\ \ldots \
$, $\alpha_m\ge -1$, $A_m\in{\mathbb R }$, $m\in{\mathbb N}$, $t\in
[T,\infty ]$, $\psi(t)\equiv\psi(T)$,  $t\in [0,T]$ with a large
enough $T\in(0,\infty)$.}
\end{corollary}

\medskip

\begin{remark}\label{remark4.1.k}
Note that the condition (\ref{eq7.77}) is not only sufficient but
also necessary for the compactness of the classes $\frak{R}_{Q}$
with integral constraints of the type (\ref{eq7.7A}) on $Q$, see
Theorem 5.1 in \cite{RS$_2$}, and (\ref{eq7.77}) is equivalent to
the following condition
\begin{equation}\label{eqKR4.4}
\int\limits_{\Delta}^{\infty}\log\,\Phi(t)\,\frac{dt}{t^{n'}}=+\infty\end{equation}
for all $\Delta>t_0$ where $t_0:=\sup\limits_{\Phi(t)=0}t,$ $t_0=0$
if $\Phi(0)>0,$ and where $\frac{1}{n'}+\frac{1}{n}=1$, i.e., $n'=2$
for $n=2$, $n'$ is strictly decreasing in $n$ and $n'=n/(n-1)\to1$
as $n\to\infty$, see Remark 4.2 in \cite{RS$_2$}.

\bigskip

Finally, all the notes in Remarks \ref{remark4.1}, \ref{remark4.3}
and \ref{remark6.1} above are also related to the compactness
results in this section.
\end{remark}

\bigskip

These results will have, in particular, wide applications to the
convergence and compactness theory for the Sobolev homeomorphisms as
well as for the Orlicz--Sobolev homeomorphisms, see e.g.
\cite{KRSS}, that will be published el\-se\-whe\-re.

\medskip
\noindent
{\bf Vladimir Ryazanov and Evgeny Sevost'yanov:}\\
Institute of Applied Mathematics and Mechanics,\\
National Academy of Sciences of Ukraine,\\ 74 Roze Luxemburg Str.,
Donetsk,\\ 83114, UKRAINE\\
e--mails: vlryazanov1@rambler.ru; \\
brusin2006@rambler.ru

\end{document}